\documentclass[a4paper]{article}

\makeatletter
\let\@fnsymbol\@arabic
\makeatother

\usepackage{hyperref} 

\usepackage[utf8]{inputenc}
\usepackage{cite}

\usepackage{amssymb,amsmath,amsthm}
\usepackage[english]{babel}

\usepackage{todonotes}
\usepackage{verbatim} 
\usepackage{enumerate}
\usepackage{mathtools}

\usepackage{a4wide}
\usepackage{multirow}
\usepackage[footnotesize,bf,centerlast]{caption}
\usepackage{xcolor,colortbl}
\definecolor{Gray}{gray}{0.80}
\definecolor{LightGray}{gray}{0.90}

\usepackage[normalem]{ulem}

\usepackage{bbm}

\newcommand{\cA}{\mathcal{A}}

\newcommand{\cC}{\mathcal{C}}
\newcommand{\cD}{\mathcal{D}}

\newcommand{\cL}{\mathcal{L}}

\newcommand{\cP}{\mathcal{P}}

\newcommand{\bR}{\mathbb{R}}

\newcommand{\bfH}{\mathbf{H}}

\newcommand{\PR}{\mathbb{P}}

\newcommand{\dd}{ \mathrm{d}}

\DeclareMathOperator*{\LIM}{LIM} 
\DeclareMathOperator*{\subLIM}{subLIM}
\DeclareMathOperator*{\superLIM}{superLIM}

\renewcommand{\epsilon}{\varepsilon}

\newcommand{\vn}[1]{\left| \! \left| #1\right| \! \right|}

\numberwithin{equation}{section}

\newtheorem{theorem}{Theorem}[section]
\newtheorem{lemma}[theorem]{Lemma}
\newtheorem{proposition}[theorem]{Proposition}

\theoremstyle{definition}
\newtheorem{definition}[theorem]{Definition}

\newtheorem*{remark*}{Remark}

\newtheorem{assumption}[theorem]{Assumption}

\newtheorem{condition}[theorem]{Condition}

\newcommand{\N}{\mathbb{N}}
\newcommand{\R}{\mathbb{R}}

\newcommand{\Gs}{G}
\newcommand{\Hs}{H}
\newcommand{\Hf}{\Gamma}
\newcommand{\Kf}{\Lambda}

\title{Path-space moderate deviations for a Curie-Weiss model of self-organized criticality} 

\author{Francesca Collet\thanks{Delft Institute of Applied Mathematics, Delft University of Technology, Mekelweg 4, 2628 CD Delft (The Netherlands). \emph{E-mail address}: f.collet-1@tudelft.nl} \and Matthias Gorny\thanks{Laboratoire de Math{\'e}matiques, Universit{\'e} Paris-Sud, Bâtiment 425, 
91405 Orsay Cedex (France). \emph{E-mail address}: matthias.gorny@math.u-psud.fr } \and Richard C. Kraaij\thanks{Fakultät für Mathematik, Ruhr-University of Bochum, Postfach 102148, 44721 Bochum (Germany). \emph{E-mail address}: Richard.Kraaij@rub.de}}

\date{}

\begin{document}

\maketitle

\begin{abstract}
	
	\noindent The dynamical Curie-Weiss model of self-organized criticality (SOC) was introduced in \cite{Gor17} and it is derived from the classical generalized Curie-Weiss by imposing a microscopic Markovian evolution having the distribution of the Curie-Weiss model of SOC \cite{CeGo16} as unique invariant measure. In the case of Gaussian single-spin distribution, we analyze the dynamics of moderate fluctuations for the magnetization. We obtain a path-space moderate deviation principle via a general analytic approach based on convergence of non-linear generators and uniqueness of viscosity solutions for associated Hamilton-Jacobi equations. Our result shows that, under a peculiar moderate space-time scaling and without tuning external parameters, the typical behavior of the magnetization is critical. \\
	
	\noindent \emph{Keywords:} moderate deviations $\cdot$  interacting particle systems $\cdot$ mean-field interaction $\cdot$ self-organized criticality $\cdot$ Hamilton–Jacobi equation $\cdot$ perturbation theory for Markov processes
\end{abstract}

\section{Introduction}

In their very well-known article \cite{BaTaWi87}, Bak, Tang and Wiesenfeld showed that certain large dynamical systems have the tendency to organize themselves into a critical state, without any external intervention. The amplification of small internal fluctuations can lead to a critical state and cause a chain reaction leading to a radical change of the system behavior. These systems exhibit the phenomenon of self-organized criticality (SOC) that since its introduction has been successfully applied to describe quite a number of
natural phenomena (e.g., forest fires, earthquakes, species evolution). Indeed, it has been conjectured that living systems self-organize by putting themselves in a state which is close to criticality.
In general, features of SOC have been observed empirically or simulated on a computer in various models; however, the mathematical analysis turns out to be extremely difficult, even for models whose definition is very simple  \cite{RaTo09,BaSn93,MeSa12}. Self-organized criticality has been reviewed in recent works \cite{Asc13,Bak96,Dha06,Pru12,Tur99}.

The simplest models exhibiting SOC are obtained by forcing standard critical transitions into a self-organized state \cite[Sect.~15.4]{Sor06}. The idea is to start with a model presenting a phase transition and to create a feedback from the configuration to the control parameters in order to converge towards a critical state. Following this guideline, Cerf and Gorny designed an interacting particle system exhibiting self-organized criticality that is as simple as possible and is amenable to a rigorous mathematical analysis: a \emph{Curie-Weiss model of SOC} \cite{CeGo16,Gor14}. They modified the equilibrium distribution associated to the generalized Curie-Weiss model (i.e., with real-valued spins \cite{ElNe78a}) by implementing an automatic control of the inverse temperature that, in the limit as the size $n$ goes to infinity, drives the system into criticality without tuning any external parameter. Under an exponential moment condition and a symmetry assumption on the spin distribution, they proved that the magnetization behaves as in the generalized Curie-Weiss model when posed at the critical point: the fluctuations are of order $n^{\frac{3}{4}}$ and have limiting law $ \nu(x) \propto \exp( -\frac{x^4}{12}) \, \dd x$. 

More recently, Gorny approached the problem from a non-equilibrium viewpoint and constructed a \emph{dynamical Curie-Weiss model of SOC} \cite{Gor17}. He considered a Markov process whose unique invariant distribution is the Curie-Weiss model of SOC and proved, in the case of Gaussian spins, that the fluctuations evolve on a peculiar space-time scale (orders $n^{\frac{3}{4}}$, $\sqrt{n} \, t$) and their limit is the solution of a ``critical'' SDE having $\nu$ as invariant measure.

The advantage of dealing with Gaussian spins is that it is possible to find a finite-dimensional order parameter to describe the system. In particular, the problem can be reduced to a bi-dimensional problem: the Langevin spin dynamics induce a Markovian evolution on the pair $((n^{-1}S_n(t),n^{-1}T_n(t)), t \geq 0)$, with $S_n := \sum_{i=1}^n X_i$ and $T_n := \sum_{i=1}^n X_i^2$, $X_i$'s being the spin values. Therefore it suffices to analyze the behaviour of the latter observable. 

\smallskip

Our purpose is to characterize \emph{path-space moderate deviations} for the dynamical model of SOC with Gaussian spins introduced in \cite{Gor17}. A moderate deviation principle is technically a large deviation principle and consists in a refinement of a central limit theorem, in the sense that it characterizes the exponential decay of the probability of deviations from the average on a smaller scale. 

We apply the approach to large deviations by Feng-Kurtz \cite{FeKu06} to characterize the most likely behavior for the trajectories of fluctuations. The techniques are based on the convergence of Hamiltonians and well-posedness of a class of Hamilton-Jacobi equations corresponding to a limiting Hamiltonian $H$. These techniques have been recently exploited to analyze moderate fluctuations from equilibrium in the various regimes in the standard \cite{CoKr17} and the random-field version \cite{CoKr18} of the Curie-Weiss model. 
The major difference in comparison to these papers is that now we are dealing with {\em unbounded spin state space}. Nevertheless, we can implement the same strategy as in \cite{CoKr18}. We use the perturbation theory for Markov processes \cite{Kur73a,Kur73b,PaStVa77} to formally identify a limiting operator $H$ and we relax our definition of limiting operator to allow for unbounded functions in the domain. More precisely, we follow \cite{FeKu06} and introduce two Hamiltonians $H_{\dagger}$ and $H_\ddagger$, that are limiting upper and lower bounds for the sequence of Hamiltonians $H_n$, respectively. We then characterize $H$ by matching the upper and lower bound. 

From a qualitative viewpoint, we derive a {\em projected} large deviation principle. Indeed, there is a natural time-scale separation for the evolutions of the two processes $(n^{-1}S_n(t), t \geq 0)$ and \mbox{$(n^{-1}T_n(t), t \geq 0)$}: $n^{-1}T_n$ is fast and converges exponentially quickly to $\sigma^2$, the variance of the single-spin distribution, while $n^{-1}S_n$ is slow and its limiting behavior can be determined after suitably ``averaging out'' the dynamics of $n^{-1}T_n$. Corresponding to this observation, we need to prove a large deviation principle for the component $n^{-1}S_n$ only. 
Our main result shows that self-organized criticality is reflected by moderate deviations, since the rate function for the path-space moderate deviation principle retains the features of the ``critical'' evolution derived in \cite{Gor17}.\\

The outline of the paper is as follows: in Section~\ref{sect:model_and_result} we formally introduce the dynamical version of the Curie-Weiss model of SOC and we state the large deviation principle. The proof is given in Section~\ref{sect:proof}. Appendix~\ref{appendix:large_deviations_for_projected_processes} contains the mathematical tools needed to derive our large deviation principle via solving a class of associated Hamilton-Jacobi equations and it is included to make the paper self-contained. A similar version of the appendix appears also in \cite{CoKr18}.

\section{Model and main result}
\label{sct:model_and_result}

\subsection{Notation and definitions} 

Before starting with the main contents of the paper, we introduce some notation. We start with the definition of good rate-function and of large deviation principle for a sequence of random variables. 
	
\begin{definition}
Let $(X_n)_{n \in \mathbb{N}^*}$ be a sequence of random variables on a Polish space $\mathcal{X}$. Furthermore, consider a function $I : \mathcal{X} \rightarrow [0,\infty]$ and a sequence $(r_n)_{n \in \mathbb{N}^*}$ of positive numbers such that $r_n \uparrow \infty$. We say that
\begin{itemize}
\item  
the function $I$ is a \textit{good rate-function} if the set $\{x \, | \, I(x) \leq c\}$ is compact for every $c \geq 0$.
\item 
the sequence $(X_n)_{n \in \mathbb{N}^*}$ is \textit{exponentially tight} at speed $r_n$ if, for every $a \geq 0$, there exists a compact set $K_a \subseteq \mathcal{X}$ such that $\limsup_n r_n^{-1} \log \, \PR[X_n \notin K_a] \leq - a$.
\item 
the sequence $(X_n)_{n \in \mathbb{N}^*}$ satisfies the \textit{large deviation principle} with speed $r_n$ and good rate-function $I$, denoted by 
\begin{equation*}
\PR[X_n \approx a] \asymp e^{-r_n I(a)},
\end{equation*}
if, for every closed set $A \subseteq \mathcal{X}$, we have 
\begin{equation*}
\limsup_{n \uparrow \infty} \, r_n^{-1} \log \PR[X_n \in A] \leq - \inf_{x \in A} I(x),
\end{equation*}
and, for every open set $U \subseteq \mathcal{X}$, 
\begin{equation*}
\liminf_{n \uparrow \infty} \, r_n^{-1} \log \PR[X_n \in U] \geq - \inf_{x \in U} I(x).
\end{equation*}
\end{itemize}
\end{definition}

\begin{definition} 
A curve $\gamma: [0,T] \to \mathbb{R}$ is absolutely continuous if there exists a function \mbox{$g \in L^1([0,T],\bR)$} such that for $t \in [0,T]$ we have $\gamma(t) = \gamma(0) + \int_0^t g(s) \dd s$. We write $g = \dot{\gamma}$. A curve $\gamma: \bR^+ \to \mathbb{R}$ is absolutely continuous if the restriction to $[0,T]$ is absolutely continuous for every $T \geq 0$. Throughout the whole paper $\cA\cC$ will denote the set of absolutely continuous curves in $\bR$.
\end{definition}

To conclude we fix notation for some collections of function-spaces. 

\begin{definition}
Let $k \geq 1$ and $E$ a closed subset of $\mathbb{R}^d$. We will denote by 
\begin{itemize}
\item
$C_l^k(E)$ (resp. $C_u^k(E)$) the set of functions that are bounded from below (resp. above) in $E$ and are $k$ times differentiable on a neighborhood of $E$ in $\mathbb{R}^d$.
\item
$C_c^k(E)$ the set of functions that are constant outside some compact set in $E$ and are $k$ times continuously differentiable on a neighborhood of $E$ in $\mathbb{R}^d$. Finally, we set \mbox{$C_c^\infty(E) := \bigcap_k C_c^k(E)$.}
\end{itemize}
\end{definition}

\subsection{Description of the model and main result}
\label{sect:model_and_result}

Let $\rho$ be a symmetric probability measure on $\mathbb{R}$, with variance $\sigma^2$, and such that we have $\int_{\mathbb{R}} \exp(az^2) \dd \rho(z) < \infty$, for every $a \geq 0$. The \emph{generalized Curie-Weiss model} associated with $\rho$ and inverse temperature $\beta >0$ is an infinite triangular array of real-valued spin random variables $(X_n^k)_{1 \leq k \leq n}$ having joint distribution
\begin{equation}\label{eqn:generalized_CW_distribution}
\dd \mu^{\mathrm{CW}}_{n,\rho,\beta} (z_1, \dots, z_n) = \frac{1}{Z_n(\beta)} \exp \left( \frac{\beta}{2} \frac{(z_1 + \cdots + z_n)^2}{n} \right) \prod_{i=1}^n \dd\rho(z_i),
\end{equation}
where $Z_n(\beta)$ is a normalizing constant. For any $n \geq 1$, set $S_n := X_n^1+\cdots+X_n^n$. We have the following results for the asymptotics of $S_n$ (cf. \cite{ElNe78a}):
\begin{itemize}
	\item 
	If $\beta < \frac{1}{\sigma^2}$, then the fluctuations of $S_n$ are of order $\sqrt{n}$ and, in particular, $\frac{S_n}{\sqrt{n}}$ converges in law to a centered Gaussian random variable with variance $\frac{\sigma^2}{1-\beta \sigma^2}$.
	\item The point $\beta = \frac{1}{\sigma^2}$ is the critical point for the system. The fluctuations of $S_n$ become of higher order and their limit is no more Gaussian. Indeed, there exist $k \in \mathbb{N} \setminus \{0,1\}$ and $\lambda > 0$  (both depending on $\rho$), such that
	\begin{equation}\label{eqn:genCW:equi_fluctuation_limit}
	\frac{S_n}{n^{1-1/2k}} \: \xrightarrow[\: n \uparrow\infty \:]{\mathcal{L}} \: C_{k,\lambda} \exp \left( -\lambda \frac{s^{2k}}{(2k)!} \right) \, \dd s,
	\end{equation}
	where $C_{k,\lambda}$ is a normalizing constant.
\end{itemize} 

In \cite{CeGo16} the authors modified the distribution \eqref{eqn:generalized_CW_distribution} so as to build a system of interacting random variables that exhibits a phenomenon of self-organized criticality. In other words, they constructed a spin system converging to the critical state of \eqref{eqn:generalized_CW_distribution} (corresponding to $\beta=\frac{1}{\sigma^2}$) without tuning any external parameter. Based on the observation that if the spins were independent the quantity $n (z_1^2+\cdots+z_n^2)^{-1}$ would be a good estimator for $\frac{1}{\sigma^2}$ by strong law of large numbers, they decided to replace the inverse temperature $\beta$ in \eqref{eqn:generalized_CW_distribution} with $n (z_1^2+\cdots+z_n^2)^{-1}$, obtaining
\begin{equation}\label{eqn:CW_of_SOC_distribution}
\dd \mu^{\mathrm{SOC}}_{n,\rho} (z_1, \dots, z_n) =  \frac{1}{Z_n} \exp \left( \frac{1}{2} \frac{(z_1 + \cdots + z_n)^2}{z_1^2+\cdots+z_n^2} \right) \prod_{i=1}^n \dd \rho(z_i),
\end{equation}
where $Z_n$ is a normalizing constant. An infinite triangular array of real-valued spins  $(X_n^k)_{1 \leq k \leq n}$ having joint distribution \eqref{eqn:CW_of_SOC_distribution} is a {\em Curie-Weiss model of self-organized criticality}  and it indeed evolves spontaneously towards criticality. The fluctuations of $S_n$, under \eqref{eqn:CW_of_SOC_distribution}, have the same asymptotics as the critical generalized Curie-Weiss model, in the sense that they obey the same result as \eqref{eqn:genCW:equi_fluctuation_limit} with a universal exponent $k=2$.

In~\cite{Gor17} a {\em dynamical version} of the Curie-Weiss model of SOC was introduced. It consists in a Markov process, defined through a system of $n$ interacting Langevin diffusions, whose unique invariant distribution is 
\begin{equation}\label{eqn:CW_of_SOC_distribution:modification}
\dd \tilde{\mu}^{\mathrm{SOC}}_{n,\rho} (z_1, \dots, z_n) = \frac{1}{Z_n} \exp \left( \frac{1}{2} \frac{(z_1 + \cdots + z_n)^2}{z_1^2+\cdots+z_n^2+1} \right) \prod_{i=1}^n \dd \rho(z_i),
\end{equation}
where $Z_n$ is a normalizing constant. Observe that \eqref{eqn:CW_of_SOC_distribution:modification} is a slight modification of \eqref{eqn:CW_of_SOC_distribution} aimed at avoiding technical difficulties due to ill-definition of the distribution at the origin and to the non-Lipschitzianity of the coefficients of the associated Langevin diffusions. Nevertheless the distributions \eqref{eqn:CW_of_SOC_distribution} and \eqref{eqn:CW_of_SOC_distribution:modification} provide two equivalent formulations for a Curie-Weiss model of SOC (see \cite{Gor17} and references therein for further details). Now we come to the description of the dynamics we are interested in.

Let $\varphi: \mathbb{R} \to \mathbb{R}$ be an even function of class $C^{2}$ such that $\exp(2\varphi)$ is integrable over $\R$. Moreover, suppose that there exists a positive constant $c$ such that, for any $z \in \mathbb{R}$, $z\varphi'(z)\leq c(1+z^2)$. We define $\rho$ to be the probability measure having density
\[
\rho(z) = \exp(2\varphi(z))\,\left(\int_{\R}\exp(2\varphi(w))\, \dd w\right)^{-1},
\]
with respect to the Lebesgue measure on $\R$. The dynamical counterpart of the Curie-Weiss model of SOC \eqref{eqn:CW_of_SOC_distribution:modification} is an infinite triangular array of stochastic processes $(X_{n}^{k}(t),\,t\geq 0)_{1\leq k \leq n}$ such that, for all $n \geq 1$, $\left((X^{1}_{n}(t),\dots,X^{n}_{n}(t)),\,t\geq 0\right)$ is the unique solution of the following system of stochastic differential equations:
\begin{equation}\label{eqn:Langevin_dynamics}
\dd X_n^j(t)= \frac{1}{2} \left[ 2 \varphi' \left( X_n^j(t) \right) + \frac{S_n(t)}{T_n(t)+1} - X_n^j(t)\left( \frac{S_n(t)}{T_n(t)+1} \right)^2 \right] \dd t + \dd B_j(t), \quad (j=1,\dots,n)
\end{equation}
where
\begin{itemize}
	\item
	for every $t \geq 0$, 
	\[
	S_n(t) := X_n^1(t)+\dots+X_n^n(t) \quad \mbox{ and } \quad T_n(t)=\left(X_n^1(t)\right)^2+\dots+\left(X_n^n(t)\right)^2;
	\]
	\item the process $((B_1(t),\dots,B_n(t)), t \geq 0)$ is a standard $n$-dimensional Brownian motion.
\end{itemize}
The solution $\left((X^{1}_{n}(t),\dots,X^{n}_{n}(t)),\,t\geq 0\right)$ of \eqref{eqn:Langevin_dynamics} is a Markov diffusion process on $\mathbb{R}^n$. For any $f\in C^2(\R^n)$ and $z \in \R^n$, it evolves with infinitesimal generator
\begin{equation}\label{eqn:configuration_inf_gen}
L_nf(z)=\frac{1}{2}\sum_{j=1}^n\frac{\partial^2 f(z)}{\partial z_j^2}+ \frac{1}{2} \sum_{j=1}^n\left[ 2\varphi'(z_j) + \frac{S_n[z]}{T_n[z]+1}-z_j\left(\frac{S_n[z]}{T_n[z]+1}\right)^2\right]\frac{\partial f(z)}{\partial z_j},
\end{equation}
with $S_n[z] := \sum_{i=1}^n z_i$ and $T_n[z] := \sum_{i=1}^n z_i^2$. We recall once more that the measure \eqref{eqn:CW_of_SOC_distribution:modification} is the unique invariant distribution for $L_n$.\\

Our main aim is to describe the limiting behavior of moderate fluctuations for the evolution \eqref{eqn:CW_of_SOC_distribution:modification}; the technical difficulties arising have not allowed us to obtained the desired results under the present assumptions, in particular with no requirements on the function $\varphi$ (except evenness and exponential integrability). Thus we find it preferable to make the following assumption at this point:

\begin{itemize}
	\item[\textbf{(A)}] $\varphi (z) = - \dfrac{z^2}{4 \sigma^2}$, for some $\sigma>0$.
\end{itemize}

Assumption \textbf{(A)} corresponds to choosing the Gaussian probability density as reference measure $\rho$ for the spin variables. Under assumption \textbf{(A)}, the process $\left( (n^{-1} S_n(t), n^{-1}T_n(t)), t \geq 0 \right)$ is a \emph{sufficient statistics} for our model. Indeed, the dynamics \eqref{eqn:configuration_inf_gen} on the configurations induce a Markovian dynamics on $\mathbb{R}^2$ for the process $\left( (n^{-1} S_n(t), n^{-1}T_n(t)-\sigma^2), t \geq 0 \right)$  that evolves with generator 
\begin{multline}\label{eqn:statistics_inf_gen}
A_n f (x,y) = \frac{1}{2n} \frac{\partial^2 f}{\partial x^2}(x,y) +\frac{2x}{n} \frac{\partial^2 f}{\partial x \partial y}(x,y) + \frac{2(y+\sigma^2)}{n} \frac{\partial^2 f}{\partial y^2}(x,y)\\
+ \frac{1}{2} \left[ -\frac{n^2 x^3}{(ny+n\sigma^2+1)^2} + \frac{nx}{ny+n\sigma^2+1} - \frac{x}{\sigma^2} \right] \frac{\partial f}{\partial x}(x,y) \\
+ \left[ \frac{nx^2}{(ny+n\sigma^2+1)^2} - \frac{y}{\sigma^2} \right] \frac{\partial f}{\partial y}(x,y). 
\end{multline}
The derivation  of the previous formula from \eqref{eqn:CW_of_SOC_distribution:modification} is omitted, since it is tedious and rather standard. We refer to \cite[Sect.~3, Prop.~6]{Gor17} for the detailed derivation of a similar result (the main difference being the space-time scaling the process is subject to).
	
As a consequence of \eqref{eqn:statistics_inf_gen}, the task of characterizing the time-evolution of the fluctuation flow
\[
\left( \frac{1}{n} \sum_{k=1}^n \delta_{X_n^k(t)} - \dd \rho(z) \right)_{t\geq0 }
\]
turns into analyzing the path-space deviations of $\left(\left(n^{-1} S_n(t), n^{-1}T_n(t) - \sigma^2\right), t \geq 0 \right)$. From being infinite dimensional, the problem reduces to a two dimensional problem.\\

First consider a standard central limit theorem setting and therefore consider the two dimensional process classically rescaled by $\sqrt{n}$. Computing the formal limit of \eqref{eqn:statistics_inf_gen} for functions of the variable $x$ (resp. $y$) only, we find that, as $n \uparrow \infty$, the process $\left(n^{-1/2} S_n(t), t \geq 0 \right)$ converges weakly to a standard Brownian motion, whereas $\left( \sqrt{n} \left(n^{-1}T_n(t) - \sigma^2\right), t \geq 0 \right)$ to the Ornstein-Uhlenbeck process solution of
	\begin{equation} \label{eqn:OU_process}
	\dd Y(t) = - \frac{Y(t)}{\sigma^2} \dd t + 2 \sigma \, \dd B_1(t),
	\end{equation} 
	where $B_1$ is a standard Brownian motion.
%
Thus, the second component of the pair $\left(\left(n^{-1/2} S_n(t), \sqrt{n} \left(n^{-1}T_n(t) - \sigma^2 \right)\right), t \geq 0 \right)$ has a confined process as a limit, whereas the first one fluctuates homogeneously in space. Indeed, in this last case, as shown in \cite{Gor17}, a further rescaling  allows one to see that the process $\left(n^{-3/4} S_n(\sqrt{n}t), t \geq 0 \right)$ converges weakly to the solution of 
	\begin{equation*}
	\dd X(t) = - \frac{X^3(t)}{2 \sigma^4} \dd t + \dd B(t),
	\end{equation*} 
with $B(t)$ standard Brownian motion.
Under this critical space-time rescaling the process, $\left( n^{-1}T_n(t) - \sigma^2, t \geq 0 \right)$ collapses: at times of order $\sqrt{n}t$ the process $ \left(\sqrt{n} \left(n^{-1}T_n(\sqrt{n} t) - \sigma^2\right), t \geq 0 \right)$ equilibrates at a Gaussian measure and therefore, when refining the space rescaling, the process $\left( n^{1/4} \left(n^{-1}T_n(\sqrt{n} t) - \sigma^2\right) \right)$ equilibrates at $\delta_{0}$. This was proven in \cite[Lem.~9]{Gor17}.
	
We complement the analysis by considering the moderate deviations of $n^{-1} S_n(t)$ around  equilibrium, under the microscopic dynamics \eqref{eqn:statistics_inf_gen}. As in the weak convergence setting mentioned above, corresponding to the separation of time-scales for the evolutions of the two processes, we need to prove a projected path-space large deviation principle, in other words for the component $n^{-1}S_n$ only. More precisely, we get the following statement.  

\begin{theorem}\label{thm:path_space_MDP}
	Let $(b_n)_{n \in \mathbb{N}^*}$ be a sequence of positive real numbers such that $b_n \uparrow \infty$ and  \mbox{$b_n^4 n^{-1} \downarrow 0$}. Suppose that $b_n n^{-1} S_n(0)$ satisfies a large deviation principle with speed $nb_n^{-4}$ on $\mathbb{R}$ and rate function $I_0$. Then, the trajectories $\left( b_n n^{-1} S_n(b_n^2 t), t \geq 0 \right)$ satisfy the large deviation principle
	\[
	\mathbb{P} \left[ \left( b_n n^{-1} S_n(b_n^2 t), t \geq 0 \right) \approx (\gamma(t), t \geq 0) \right] \asymp e^{-nb_n^{-4} I(\gamma)}
	\]
	on $C_{\mathbb{R}}(\mathbb{R}^+)$, with good rate function
	\begin{equation}\label{eqn:rate_function}
	I(\gamma) = 
	\begin{cases}
	I_0(\gamma(0)) + \int_0^{+\infty} \cL(\gamma(s),\dot{\gamma}(s))\, \dd s & \text{ if } \gamma \in \mathcal{AC},\\
	\infty & \text{ otherwise},
	\end{cases}
	\end{equation}
	where
	\[
	\mathcal{L}(x,v) := \frac{1}{2} \left\vert v + \frac{x^3}{2\sigma^4} \right\vert^2.
	\]
\end{theorem}

By choosing the sequence $b_n = n^{\alpha}$, with $\alpha > 0$, we can rephrase Theorem~\ref{thm:path_space_MDP} in terms of more familiar moderate scalings involving powers of the system-size. We therefore get estimates for the probability of a typical trajectory on a scale that is between a law of large numbers and a central limit theorem.  
This result extends our understanding of the path-space fluctuations for the Curie-Weiss model of self-organized criticality, in the case of Gaussian spins. We have stated this result, in combination with the non-standard central limit theorem in \cite[Thm.~1]{Gor17} in Table~\ref{tab:SOC_CW_deviations}.
The displayed conclusions are drawn under the assumption that in each case the initial condition satisfies a large deviation principle at the correct speed. Observe that self-organized criticality is reflected by moderate deviations, since the rate function retains the features of the ``critical'' evolution \eqref{eqn:critical_SDE}.  
To conclude, it is worth to mention that the methods of the papers \cite{CoKr17,CoKr18} are not sufficient to obtain a path-space large deviation principle for the process $((n^{-1}S_n(t),n^{-1}T_n(t), t \geq 0)$ by the Feng-Kurtz approach. Indeed, the Hamiltonian is not of the standard type dealt with in \cite{CoKr17} and it is not  immediately clear how the comparison principle can be treated.

\begin{table}[h!]
	\caption{Path-space fluctuations for the magnetization of the Curie-Weiss model of self-organized criticality in the case of Gaussian spins}
	\begin{center}
		\begin{tabular}{|c|c|c|}
			\hline
			\rowcolor{LightGray}
			\parbox[c][1cm][c]{1.5cm}{\scshape \footnotesize \centering Scaling Exponent} & \parbox[c][1cm][c]{4cm}{\scshape \footnotesize \centering Rescaled Process} & \parbox[c][1cm][c]{4cm}{\scshape \footnotesize \centering Limiting Theorem}\\
			\hline \hline
			$\alpha \in \left( 0, \frac{1}{4} \right)$ & {\footnotesize $\left( n^{\alpha-1} \, S_n \left( n^{2\alpha}t \right), t \geq 0 \right)$} & \parbox[c][1.5cm][c]{3cm}{\scriptsize \centering LDP at speed $n^{1-4\alpha}$ with rate function \eqref{eqn:rate_function}}\\
			\hline
			$\alpha = \frac{1}{4}$ & {\footnotesize $\left( n^{-3/4} \, S_n \left( n^{1/2}t \right), t \geq 0 \right)$} & \parbox[c][3cm][c]{5.8cm}{\scriptsize \centering  weak convergence to the unique solution of 
				\begin{equation}\label{eqn:critical_SDE}\tag{$\star$}
				\dd X(t) = - \dfrac{X^3(t)}{2\sigma^4} \, \dd t + \dd B(t) 
				\end{equation} 
				with initial condition \\[0.1cm] $X(0) = 0$ \\[0.1cm] (see \cite[Thm.~1]{Gor17})} \\
			\hline 
		\end{tabular}
	\end{center}
	\label{tab:SOC_CW_deviations}
\end{table}

\section{Proof}
\label{sect:proof}

We aim at studying moderate deviations by following the methods in \cite{FeKu06}. The techniques are based on the convergence of Hamiltonians and well-posedness of a class of Hamilton-Jacobi equations corresponding to a limiting Hamiltonian. These techniques have been applied also in \cite{DFL11,FeFoKu12,Kr16b,CoKr17,CoKr18}.  
In particular, in \cite{CoKr18} moderate deviation principles for projected processes are proved by combining the perturbation theory for Markov processes with a sophisticated notion of convergence of Hamiltonians, based on limiting upper and lower bounds.  Here we apply those same techniques, as they allow to take care of {\em unbounded spin state space.} We summarize the notions needed for our result and the abstract machinery used for the proof of a large deviation principle via well-posedness of Hamilton-Jacobi equations in Appendix \ref{appendix:large_deviations_for_projected_processes}. We rely on Theorem \ref{theorem:Abstract_LDP} for which we must check the following conditions:

\begin{itemize}
\item 
The processes $\left( \left(b_n n^{-1} S_n(b_n^2t), b_n \left( n^{-1} T_n(b_n^2 t) - \sigma^2 \right) \right), t \geq 0 \right)$  satisfy an appropriate exponential compact containment condition. See Section~\ref{Subsct:exponential_compact_containment}.
\item
There exist two Hamiltonians $H_\dagger \subseteq C_l(\mathbb{R}^2) \times C_b(\mathbb{R}^2)$ and $H_\ddagger \subseteq C_u(\mathbb{R}^2) \times C_b(\mathbb{R}^2)$ such that $H_\dagger \subseteq ex-\subLIM_n H_n$ and $H_\ddagger \subseteq ex-\superLIM_n H_n$. This extension allows for unbounded functions in the domain. See Section~\ref{subsct:perturbation_and_approximation}. Moreover, we refer to Definition~\ref{definition:subLIM_superLIM} for the notions of $\subLIM$ and $\superLIM$.
\item 
There is an operator $H \subseteq C_b(\bR) \times C_b(\bR)$ such that, for all $\lambda > 0$ and $h \in C_b(\mathbb{R})$, every viscosity subsolution to $f- \lambda H_\dagger f = h$ is a viscosity subsolution to $f - \lambda Hf = h$ and  every viscosity supersolution to $f - \lambda H_\ddagger f = h$ is a viscosity supersolution to $f - \lambda H f = h$. The operators $H_\dagger$ and $H_\ddagger$ should be thought of as upper and lower bounds for the ``true'' limiting $H$ of the sequence $H_n$. See Section~\ref{subsct:perturbation_and_approximation}. 
\item 
The comparison principle holds for the Hamilton-Jacobi equation $f - \lambda H f = h$ for all $h \in C_b(\mathbb{R})$ and all $\lambda>0$. The proof of this statement is immediate, since the operator $H$ we will be dealing with is of the type considered in \cite{CoKr17}.
\end{itemize}

For the verification of all the open conditions we use the limiting behaviour of the sequence of Hamiltonians $H_n$. We then start by deriving an expansion for the Hamiltonians associated to the re-scaled fluctuation process.

\subsection{Expansion of the Hamiltonian}

Let $(b_n)_{n \in \mathbb{N}^*}$ be a sequence of positive real numbers such that $b_n \uparrow \infty$ and $b_n^4 n^{-1} \downarrow 0$. The fluctuation process $\left( \left(b_n n^{-1} S_n(b_n^2t), b_n \left( n^{-1} T_n(b_n^2 t) - \sigma^2 \right) \right), t \geq 0 \right)$ has Markovian  evolution on state space \mbox{$E_n := \mathbb{R} \times  (-\sigma^2 b_n,+\infty)$} and its generator $G_n$ can be deduced from \eqref{eqn:statistics_inf_gen}. 

\begin{lemma}\label{lemma:space-time_rescaled_process_generator}
Let $n\in \N^*$. The Markov process $\left( \left(b_n n^{-1} S_n(b_n^2t), b_n \left( n^{-1} T_n(b_n^2 t) - \sigma^2 \right) \right), t \geq 0 \right)$  has infinitesimal generator $\Gs_n$ that, for any $f \in C^2_c(E_n)$, satisfies 
\begin{multline}\label{eqn:space-time_rescaled_process_generator}
\Gs_nf(x,y)= \frac{1}{2} \left(\dfrac{xb_n^2}{\sigma^2}(h_n(y)-1)-\dfrac{x^3}{\sigma^4}h_n^2(y)\right)\frac{\partial f}{\partial x}(x,y)+\left(\dfrac{b_nx^2}{n\sigma^4}h_n^2(y)-\dfrac{b_n^2y}{\sigma^2}\right)\frac{\partial f}{\partial y}(x,y)\\
+\dfrac{b_n^4}{2n}\frac{\partial^2 f}{\partial x^2}(x,y) +\dfrac{2 b_n^3x}{n}\frac{\partial^2 f}{\partial x\partial y}(x,y)+\dfrac{2b_n^4}{n}\left(\dfrac{y}{b_n}+\sigma^2\right)
\frac{\partial^2 f}{\partial y^2}(x,y),
\end{multline}
where the function $h_n: \, (-\sigma^2 b_n,+\infty) \, \to \mathbb{R}$ is defined by $h_n(y) = \left( 1+ \frac{y}{b_n\sigma^2}+\frac{1}{n\sigma^2}\right)^{-1}$.
\end{lemma}

By applying the chain rule to the function $\exp \{ nb_n^{-4}f(x,y) \}$, for $f \in C^2_c(E_n)$, it is easy to see that, at speed $nb_n^{-4}$, the Hamiltonian
\[
H_n f (x,y) = b_n^4 n^{-1} e^{-nb_n^{-4} f(x,y)} G_n \left( e^{nb_n^{-4} f} \right) (x,y)
\]
results in
\begin{multline}\label{eqn:Hamiltonian}
\Hs_nf(x,y)=\Gs_nf(x,y)+\dfrac{1}{2}\left(\dfrac{\partial f}{\partial x}(x,y)\right)^2+2\sigma^2\left(\dfrac{\partial f}{\partial y}(x,y)\right)^2
\\
+\dfrac{2 x }{b_n}\dfrac{\partial f}{\partial x}(x,y)\dfrac{\partial f}{\partial y}(x,y)+\dfrac{2y }{b_n}
\left(\frac{\partial f}{\partial y}(x,y)\right)^2,
\end{multline}
with $G_n$ given by \eqref{eqn:space-time_rescaled_process_generator}. We Taylor expand the function $h_n(y)$ appearing in the definition of $G_n$ up to second order:
\begin{equation}\label{eqn:Taylor_expansion_h}
h_n(y)=1-\frac{y}{b_n\sigma^2}+\frac{y^2}{b_n^2\sigma^4}+\frac{1}{b_n^2}\varepsilon_n(y),
\end{equation}
where the sequence of functions $(\varepsilon_n)_{n\in \N^*}$ converges to zero, uniformly in $n$, on compact sets of~$\R$. 

In what follows we will require a more accurate control on the reminder $\varepsilon_n(y)$. For this reason we give here the following lemma.

\begin{lemma} \label{lemma:control_epsilon} 
Set $K_n = [-\sigma^2 \log^{1/2} b_n^{1/2},\sigma^2 \log^{1/2} b_n^{1/2}]$. There exists a positive constant $c$, independent of $n$, such that we have
\begin{equation}\label{eqn:remainder_Taylor_expansion_h}
\sup_{y \in K_n} |\epsilon_n(y)| \leq c \, b_n^{-1} \log^{3/2} b_n^{1/2}.
\end{equation}
\end{lemma} 

\begin{proof}
We Taylor expand the function $h_n(y)$ up to second order and we express the reminder in Lagrange's form. Taking out the highest order terms to obtain \eqref{eqn:Taylor_expansion_h}, we find
\begin{multline*}
\epsilon_n(y) = - \frac{b_n^2}{1 + n\sigma^2} - \frac{b_n y}{\sigma^2} \left[\left( 1 - \frac{1}{1 + n\sigma^2} \right)^2 - 1 \right] +  \frac{y^2}{\sigma^4} \left[\left( 1 - \frac{1}{1 + n\sigma^2} \right)^3 - 1 \right] \\
- \frac{y^3}{b_n \sigma^6} \left(1 + \frac{\zeta}{b_n \sigma^2} + \frac{1}{n\sigma^2}\right)^{-4},
\end{multline*}
with $|\zeta| < |y|$. Note that the first three terms on the right-hand side are at most of order $b_n^{-2}$. The final term is of order $b_n^{-1} \log^{3/2} b_n^{1/2}$, as the fraction that is taken to the fourth power is asymptotically converging to $1$. 	
\end{proof}

Turning back to the expansion of $G_n$ in \eqref{eqn:Hamiltonian}, by \eqref{eqn:Taylor_expansion_h} we get
\begin{multline}\label{eqn:Hamiltonian_expansion}
\Hs_nf(x,y) = \frac{1}{2} \left(-\frac{b_n xy}{\sigma^4} + \frac{x y^2}{\sigma^6} - \dfrac{x^3}{\sigma^4}\right)\frac{\partial f}{\partial x}(x,y) -\dfrac{b_n^2y}{\sigma^2}\frac{\partial f}{\partial y}(x,y) \\
+\dfrac{1}{2}\left(\dfrac{\partial f}{\partial x}(x,y)\right)^2+2\sigma^2\left(\dfrac{\partial f}{\partial y}(x,y)\right)^2 + R^f_n(x,y)
\end{multline}
and the remainder
\begin{multline}\label{eqn:Hamiltonian_expansion_remainder}
R^f_n(x,y) = \left(\dfrac{x \varepsilon_n(y)}{2\sigma^2}-\dfrac{x^3}{2\sigma^4} \left( h_n^2(y)-1 \right) \right)\frac{\partial f}{\partial x}(x,y) + \dfrac{b_n x^2 h_n^2(y)}{n\sigma^4} \, \frac{\partial f}{\partial y}(x,y) \\
+ \dfrac{b_n^4}{2n}\frac{\partial^2 f}{\partial x^2}(x,y)+ 2 \left(\frac{b_n^4 \sigma^2}{n} + \frac{b_n^3y}{n} \right) \frac{\partial^2 f}{\partial y^2}(x,y) + \frac{2b_n^3x}{n} \frac{\partial^2 f}{\partial x\partial y}(x,y)\\
+\dfrac{2 x }{b_n}\dfrac{\partial f}{\partial x}(x,y)\dfrac{\partial f}{\partial y}(x,y)+\dfrac{2y }{b_n}
\left(\frac{\partial f}{\partial y}(x,y)\right)^2
\end{multline}

converges to zero, uniformly in $n$, on compact sets of $\mathbb{R}^2$.

\subsection{Perturbative approach and approximating Hamiltonians}
\label{subsct:perturbation_and_approximation}

Observe that the expansion \eqref{eqn:Hamiltonian_expansion} is diverging and, more precisely, is diverging through terms containing the $y$ variable, thus relative to the time-evolution of the process $\left(b_n \left( n^{-1}T_n(b_n^2 t) - \sigma^2 \right), t \geq 0 \right)$. Indeed, the two components of $\left(\left(b_n n^{-1} S_n(b_n^2t), b_n \left( n^{-1}T_n(b_n^2 t) - \sigma^2 \right)\right), t \geq 0 \right)$ live on two different time-scales and the asymptotic behavior of \mbox{$(b_n n^{-1} S_n(b_n^2t), t \geq 0)$} can be determined after having averaged out the evolution of $\left(b_n \left( n^{-1}T_n(b_n^2 t) - \sigma^2 \right), t \geq 0 \right)$. The ``averaging'' is obtained through a perturbative approach leading to a projected large deviation principle. This argument takes inspiration from the perturbation theory for Markov processes applied in \cite{Kur73a,Kur73b,PaStVa77} and it was also used to study path-space moderate deviations for the Curie-Weiss model with random field in \cite{CoKr18}.

\smallskip

In the present section we will first give some heuristics about the perturbative method, since it will provide  guideline for getting the approximating Hamiltonians $H_\dagger, H_\ddagger$, and then we will make it rigorous.

\paragraph{Heuristics on perturbation.} In the expansion \eqref{eqn:Hamiltonian_expansion} the leading term is of order $b_n^2$ and thus explodes as $n \uparrow \infty$. We think of $b_n^{-1}$ as a perturbative parameter and we use a second order perturbation $F_{n,f}$ of $f$ to introduce some negligible (in the infinite volume limit) terms providing that the whole expansion does not diverge. \\
More precisely, given two arbitrary functions $\Hf_f, \Kf_f: \mathbb{R}^2 \to \mathbb{R}$, we define the perturbation of $f$ as
\begin{equation}\label{eqn:perturbation_definition}
F_{n,f}:(x,y)\longmapsto f(x)+ b_n^{-1} \, \Hf_f(x,y)+ b_n^{-2} \, \Kf_f(x,y)
\end{equation}
and then we choose $\Hf_f$ and $\Kf_f$ so that
\[
\Hs_n F_{n,f}(x,y) =  \Hs f(x) + \mbox{remainder},
\]
where $H f(x)$ is of order $1$ with respect to $b_n$ and the remainder contains smaller order terms. We assume that  $\Hf_f$ and $\Kf_f$ are at least of class $C^2$ and we compute $\Hs_n F_{n,f}$. Using \eqref{eqn:Hamiltonian_expansion} yields
\begin{multline*}
\Hs_nF_{n,f}(x,y)=-\frac{b_n xy}{2\sigma^4}f'(x)+\left(\frac{x y^2}{2\sigma^6}-\dfrac{x^3}{2\sigma^4}\right)f'(x)+\dfrac{1}{2}\left(f'(x)\right)^2\\
-\dfrac{yb_n}{\sigma^2}\frac{\partial \Hf_f}{\partial y}(x,y)-\frac{xy}{2\sigma^4}\frac{\partial \Hf_f}{\partial x}(x,y)-\dfrac{y}{\sigma^2}\frac{\partial \Kf_f}{\partial y}(x,y)+\mbox{remainder}.
\end{multline*}
To eliminate the terms of order $b_n$ and of order $1$ in the variable $y$, the functions $\Hf_f$ and $\Kf_f$ must necessarily verify
\begin{equation}\label{eqn:conditions_perturbative_functions}
\forall (x,y)\in \R^2
\qquad
\left\{
\begin{array}{l}
-\dfrac{y}{\sigma^2}\dfrac{\partial \Hf_f}{\partial y}(x,y)-\dfrac{xy}{2\sigma^4}f'(x)=0 \\[0.4cm]
-\dfrac{y}{\sigma^2}\dfrac{\partial \Kf_f}{\partial y}(x,y) -\dfrac{xy}{2\sigma^4}\dfrac{\partial \Hf_f}{\partial x}(x,y) + \dfrac{x y^2}{2\sigma^6} f'(x) = 0.
\end{array}
\right.
\end{equation}
If we take
\begin{equation}\label{eqn:perturbation_functions}
\Hf_f:(x,y)\longmapsto-\frac{xy}{2\sigma^2}f'(x)  \quad \mbox{ and } \quad \Kf_f:(x,y)\longmapsto \frac{xy^2}{8\sigma^4}(3f'(x)+xf''(x)),
\end{equation}
then the conditions \eqref{eqn:conditions_perturbative_functions} are satisfied and we obtain
\[
H_n F_{n,f} (x,y) = -\dfrac{x^3}{2\sigma^4}f'(x)+\dfrac{1}{2}\left(f'(x)\right)^2 + \mbox{remainder}.
\]
Provided we can control the remainder, for any function $f$ in a suitable regularity class, we {\em formally} get the following candidate limiting operator
\begin{equation}\label{eqn:limiting_Hamiltonian} 
\Hs f(x)=-\dfrac{x^3}{2\sigma^4}f'(x)+\dfrac{1}{2}\left(f'(x)\right)^2.
\end{equation}
To rigorously conclude that the Hamiltonian $H$ is the limit of the sequence $(H_n)_{n \in \mathbb{N}^*}$, with $H_n$ given in \eqref{eqn:Hamiltonian_expansion}, we should prove that $H \subseteq \LIM_n H_n$ (see Definition~\ref{def:extended_limit}). The proof of the latter assertion would consist in showing that, for every $f \in C_c^4(\mathbb{R})$, we have $\LIM_n F_{n,f} = f$ and $\LIM_n H_n F_{n,f} = Hf$. Recall that in our setting $(x,y) \in E_n = \mathbb{R} \times (-\sigma^2b_n,+\infty)$. Therefore, the functions $\Gamma_f$ and $\Lambda_f$ in \eqref{eqn:perturbation_functions}  are unbounded in $E_n$, implying in turn that also $F_{n,f}$ is unbounded in $E_n$. Due to this unboundedness, even if $f \in C_c^4(\mathbb{R})$, we can not guarantee $\sup_n \| F_{n,f} \| < \infty$ and thus we can not prove $\LIM_n F_{n,f} = f$. 

We apply the same techniques as in \cite{CoKr18}. To circumvent the problem and allow for unbounded functions in the domain, we relax our definition of limiting operator. In particular, we introduce two limiting Hamiltonians $H_{\dagger}$ and $H_\ddagger$, approximating $H$ from above and below respectively, and then we characterize $H$ by matching  upper and lower bound.

\paragraph{Approximating Hamiltonians and domain extensions.}

We have seen that the natural perturbations of our functions $f$ are unbounded. We repair this unboundedness by cutting off the functions. To this purpose, we introduce a collection of smooth increasing functions $\chi_{n} : \bR \rightarrow \bR$ such that
\begin{equation}\label{eqn:cut-off}
\chi_n(z) = 
\begin{cases}
 -\sigma^2 \log b_n^{1/2} + 1 & \text{ if } z \leq - \sigma^2 \log b_n^{1/2} \\[0.2cm]
z & \text{ if } -\sigma^2 \log b_n^{1/2} + 2 \leq z \leq \sigma^2 \log b_n^{1/2} - 2 \\[0.2cm]
\sigma^2 \log b_n^{1/2} - 1  & \text{ if } z \geq \sigma^2 \log b_n^{1/2}.
\end{cases}
\end{equation}
To make sure that the cut-off acts only outside a compact set, we first perturb our function $f$ by a Lyapunov function $\varepsilon (y^2 + \log(1+x^2))$. The latter function will indeed play a special role in establishing the exponential compact containment condition in Section \ref{Subsct:exponential_compact_containment} below.

\begin{lemma}\label{lmm:cutoff_functions}
Let $\varepsilon \in (0,1)$ and $f \in C_c^4(\mathbb{R})$. Consider the cut-off \eqref{eqn:cut-off} and define the functions
\[
\chi_n \left( F_{n,f}(x,y) \pm \varepsilon (y^2 + F_{n,g}(x,y))\right), 
\]
with $F_{n,\bullet}$ as in \eqref{eqn:perturbation_definition}, \eqref{eqn:perturbation_functions} and $g(x) = \log(1+x^2)$. Then, 
\begin{enumerate}[(a)]
\item \label{item_1:lmm:cutoff_functions}
For any $C > 0$ there is an $N = N(C)$ such that, for any $n \geq N$, we have
\begin{equation*}
\chi_n \left( F_{n,f}(x,y) \pm \varepsilon (y^2 + F_{n,g}(x,y)) \right) = F_{n,f}(x,y) \pm \varepsilon (y^2 + F_{n,g}(x,y))
\end{equation*}
on the set $K_1 = K_1(C) := \left\{(x,y) \in \bR^2 \, \middle| \, \varepsilon (y^2 + \log(1+x^2))\leq C \right\}$.
\item \label{item_2:lmm:cutoff_functions}
%
Let $\overline{C}$ be the positive constant defined in \eqref{eqn:c_bar} and set $N_1 := \sup\{n \in \mathbb{N} \, \vert \, \varepsilon \leq 6 \overline{C} b_n^{-2}\}$. Then, for any $n > N_1$, the function $\chi_n \big( F_{n,f}(x,y) \pm \varepsilon (y^2 + F_{n,g}(x,y))\big)$ is constant outside the compact set $K_{2,n} := \left\{(x,y) \in \mathbb{R}^2 \, \middle| \, \varepsilon (y^2 + \log (1+x^2)) \leq 2 \sigma^2 \log b_n^{1/2} + 6 \overline{C}  \right\}$.
%
\end{enumerate}
\end{lemma}

\begin{proof}
We start by proving \eqref{item_1:lmm:cutoff_functions}. Recall that $f \in C_c^4(\mathbb{R})$, so there exists a positive constant $M$ such that the derivatives of $f$ (and, as a consequence, $\Gamma_f$ and $\Lambda_f$) vanish at $x \notin [-M,M]$. Therefore, it yields 
%
\[
\vert f(x) \vert + b_n^{-1} \left\vert \Gamma_{f} (x, y) \right\vert + b_n^{-2} \left\vert \Lambda_{f} (x, y) \right\vert \leq \|f\| + \frac{M}{2\sigma^2} \|f'\| |y| + \frac{M}{8\sigma^4}(3\|f'\| + M\|f''\|) y^2,
\] 
where $\| \cdot \|$ denotes the $L^\infty$-norm. Moreover, since $|xg'(x)| \leq 1$ and $|x(3g'(x) + xg''(x))| \leq 8$  we also get the bound
%
%
%
\[
\varepsilon \left( b_n^{-1} \left\vert \Gamma_{g} (x, y) \right\vert + b_n^{-2} \left\vert \Lambda_{g} (x, y) \right\vert \right) \leq \frac{\vert y \vert}{\sigma^2} + \frac{y^2}{\sigma^4}.
\]
Setting 
\begin{equation}\label{eqn:c_bar}
\overline{C} :=  \max \left\{ \|f\|, \, \frac{1}{2\sigma^2}(M \|f'\|+2), \, \frac{1}{8\sigma^4}(3 M \|f'\| + M^2 \|f''\| + 8) \right\}
\end{equation}
and putting the two previous estimates together we obtain
\begin{equation}\label{estimate:lmm:cutoff_functions}
\left\vert F_{n,f}(x,y) \pm \varepsilon \left( b_n^{-1} \Gamma_{g} (x, y) + b_n^{-2}\Lambda_{g} (x, y) \right) \right\vert \leq \overline{C} \left( 1+ b_n^{-1} \vert y \vert + b_n^{-2} y^2 \right) \leq 3 \overline{C} \left( 1 +  b_n^{-2} y^2 \right),
\end{equation}
%
where the last inequality follows from $|ab| \leq 2(a^2+b^2)$, with $a,b \in\mathbb{R}$.
%
Consider an arbitrary $C > 0$. By \eqref{estimate:lmm:cutoff_functions}, we find that $(x,y) \mapsto F_{n,f}(x,y) \pm \varepsilon (y^2 + F_{n,g}(x,y))$ is bounded uniformly in $n$ on the set $K_1$. To conclude, simply observe that, since the cut-off is moving to infinity, for sufficiently large $n$, we obtain $\chi_n \equiv \mathrm{id}$ on $K_1$.

We proceed with the proof of (b). For any $n > N_{1}$ and  any $(x,y) \notin K_{2,n}$, we obtain 
\begin{align*}
F_{n,f}(x,y) &+ \varepsilon ( y^2 + F_{n,g}(x,y) ) \\
&  =   F_{n,f}(x,y) + \varepsilon \left( b_n^{-1} \Gamma_{g} (x, y) + b_n^{-2}\Lambda_{g} (x, y) \right) + \varepsilon \left(y^2 + \log(1+x^2)\right) \\
&  \geq - 3\overline{C}\left( 1+ b_n^{-2} y^2 \right) + \frac{\varepsilon}{2} y^2 + \frac{\varepsilon}{2} \left(y^2 + 2\log(1+x^2)\right) \\
& \geq - 3 \overline{C} +  \frac{\varepsilon}{2} (y^2+2\log(1+x^2)) \\
& > \sigma^2 \log b_n^{1/2}.
\end{align*}
The definition \eqref{eqn:cut-off} of the cut-off leads then to the conclusion. The proof for the function $F_{n,f}(x,y) - \varepsilon ( y^2+F_{n,g}(x,y))$ follows similarly.

\end{proof}

Before stating the next lemma, we want to make a remark on the notation $N_{\star}$ used therein. This index is an explicit positive integer larger than $N_1$ introduced in Lemma~\ref{lmm:cutoff_functions}\eqref{item_2:lmm:cutoff_functions} and it will be defined precisely in \eqref{def:N_star} at the end of this section.

\begin{lemma}\label{lemma:properties_fepsilon}
Let $\varepsilon \in (0,1)$ and $f \in C_c^4(\mathbb{R})$. Consider the cut-off \eqref{eqn:cut-off} and  define the functions
\[
f_n^{\varepsilon,\pm} (x,y) :=
\begin{cases}
0 & \text{ if } n \leq N_{\star} \\[0.1cm]
\chi_n \big( F_{n,f}(x,y) \pm \varepsilon (y^2 + F_{n,g}(x,y))\big) & \text{ if } n > N_{\star}
\end{cases}
\]
and
\[
f^{\varepsilon,\pm} (x,y) := f(x) \pm \varepsilon \left( y^2 + g(x) \right),
\]
with $F_{n,\bullet}$ as in \eqref{eqn:perturbation_definition}, \eqref{eqn:perturbation_functions} and $g(x) = \log(1+x^2)$. Then, for every $\varepsilon \in (0,1)$, the following properties are satisfied:
\begin{enumerate}[(a)]
\item \label{lemma:item:fepsilon_to_a}
$f_n^{\varepsilon,\pm} \in \cD(H_n)$.
\item \label{lemma:item:fepsilon_to_b}
$f^{\varepsilon,+} \in C_l(\bR^2)$ and $f^{\varepsilon,-} \in C_u(\bR^2)$.
\item \label{lemma:item:fepsilon_to_c}
We have 
\[
\inf_n  \inf_{(x,y) \in E_n} f_n^{\varepsilon,+}(x,y) > - \infty \quad \mbox{ and } \quad \sup_n \sup_{(x,y) \in E_n} f_n ^{\varepsilon,-}(x,y) < \infty.
\]
\item \label{lemma:item:fepsilon_to_d}
For every compact set $K \subseteq \bR^2$, there exists a positive integer $N = N(K)$ such that, for $n \geq N$ and $(x,y) \in K$, we have 
\[
f_n^{\varepsilon,\pm}(x,y) =  F_{n,f}(x,y) \pm \varepsilon  \left( y^2 + F_{n,g}(x,y) \right).
\]
\item \label{lemma:item:fepsilon_to_e} 
For every $c \in \bR$, we have
\[
\LIM_{n \uparrow \infty}  f_n^{\varepsilon,+} \wedge c  = f^{\varepsilon,+} \wedge c \quad \mbox{ and } \quad \LIM_{n \uparrow \infty}  f_n^{\varepsilon,-} \vee c = f^{\varepsilon,-} \vee c.
\]
\end{enumerate}
Moreover, it holds 
{\begin{enumerate}[(a)]\setcounter{enumi}{5}
\item \label{lemma:item:fepsilon_to_f} 
For every $c \in \bR$, we have 
\begin{equation*}
\lim_{\varepsilon \downarrow 0} \vn{f^{\varepsilon,+} \wedge c - f \wedge c} + \vn{f^{\varepsilon,-} \vee c - f \vee c} = 0.
\end{equation*} 
\end{enumerate}}
\end{lemma}
	
\begin{proof}
If $n < N_{\star}$ all the statements are trivial. We assume $n \geq N_{\star}$ and we prove all the properties for the `$+$' superscript case, the other being similar. 
\begin{enumerate}[(a)]
\item
It is clear from the definition \eqref{def:N_star} of $N_\star$ that $N_\star \geq N_1$. Then, as the cut-off \eqref{eqn:cut-off} is smooth, Lemma~\ref{lmm:cutoff_functions}\eqref{item_2:lmm:cutoff_functions} yields $f_n^{\varepsilon,\pm} \in C_c^\infty(\bR^2)$. In addition, the location of the cut-off and Lemma~\ref{lmm:cutoff_functions}\eqref{item_2:lmm:cutoff_functions} make sure that $f_n^{\varepsilon,\pm}$ is constant outside a compact set $K \subset E_n$, implying $f_n^{\varepsilon,\pm} \in \cD(G_n)$ and, as a consequence, $f_n^{\varepsilon,\pm} \in \cD(H_n)$. See equations \eqref{eqn:space-time_rescaled_process_generator} and \eqref{eqn:Hamiltonian} for the definitions of $G_n$ and $H_n$ respectively.
\item
This is immediate from the definitions of $f^{\varepsilon,\pm}$.
\item
From the estimate \eqref{estimate:lmm:cutoff_functions}, we deduce (keeping the same notation)
%
\[
\inf_{(x,y) \in \mathbb{R}^2} F_{n,f}(x,y) + \varepsilon \left( y^2 + F_{n,g}(x,y) \right) \geq - 3 \overline{C} (1 + b_n^{-2} y^2) + \varepsilon \left( y^2 + \log(1+x^2) \right),
\]
which is bounded from below uniformly in $n > N_1$. The conclusion follows as $N_\star \geq N_1$ (cf. equation \eqref{def:N_star}). 
%
\item
This follows immediately by Lemma~\ref{lmm:cutoff_functions}\eqref{item_1:lmm:cutoff_functions}.
\item
Fix $\varepsilon > 0$ and $c \in \bR$. By \eqref{lemma:item:fepsilon_to_c}, the sequence $\left( f_n^{\varepsilon,+} \wedge c \right)_{n \in \mathbb{N}^*}$ is uniformly bounded from below and then, we obviously get $\sup_{n \in \mathbb{N}^*} \vn{f_{n}^{\varepsilon,+} \wedge c} < \infty$. Thus, it suffices to prove uniform convergence on compact sets. Let us consider an arbitrary sequence $(x_n,y_n)$ converging to $(x,y)$ and prove $\lim_n f_n^{\varepsilon,+}(x_n,y_n) = f^{\varepsilon,+}(x,y)$. As a converging sequence is bounded, it follows from \eqref{lemma:item:fepsilon_to_d} that, for sufficiently large $n$, we have 
\begin{equation*}
f_n^{\varepsilon,+}(x_n,y_n) = F_{n,f} (x_n,y_n) + \varepsilon \left( y_n^2 + F_{n,g}(x_n,y_n) \right),
\end{equation*}
which indeed converges to $f^{\varepsilon,+}(x,y)$ as $n \uparrow \infty$.
\item
This follows similarly as in the proof of \eqref{lemma:item:fepsilon_to_e}.
\end{enumerate}		
\end{proof}

\begin{definition}\label{definition:limiting_upper_lower_hamiltonians}
Let $H \subseteq C_b(\bR) \times C_b(\bR)$, with domain $\cD(H) = C_c^\infty(\bR)$, be defined as 
\[
Hf(x) = - \frac{x^3}{2\sigma^4} f'(x) + \frac{1}{2} \left( f'(x) \right)^2.
\]
We define the approximating Hamiltonians $H_\dagger \subseteq C_l(\bR^2) \times C_b(\bR^2)$ and $H_\ddagger \subseteq C_u(\bR^2) \times C_b(\bR^2)$ as
\begin{align*}
H_\dagger & := \left\{ \Big(  f(x) + \varepsilon \left( y^2 + \log(1+x^2) \right), \, Hf(x) + \tfrac{\varepsilon }{2} \vn{f'} + \varepsilon^2 \Big) \, \middle| \, f \in C_c^\infty(\mathbb{R}), \varepsilon \in (0,1) \right\}, \\[0.1cm]
H_\ddagger & := \left\{ \Big(  f(x) - \varepsilon \left( y^2 + \log(1+x^2) \right), \, Hf(x) - \tfrac{\varepsilon}{2}  \vn{f'} - \varepsilon^2 \Big) \, \middle| \, f \in C_c^\infty(\mathbb{R}), \varepsilon \in (0,1) \right\}.
\end{align*}
\end{definition}	
	
\begin{proposition} \label{proposition:limiting_hamiltonians}
Consider notation as in Definitions~\ref{definition:limiting_upper_lower_hamiltonians} and \ref{definition:subLIM_superLIM}. We have $H_\dagger \subseteq ex-\subLIM_n  H_n$ and $H_\ddagger \subseteq ex-\superLIM_n  H_n$.
\end{proposition}

\begin{proof}
We prove only the first statement, i.e. $H_\dagger \subseteq ex-\subLIM_n  H_n$. Fix $\varepsilon > 0$ and $f \in C_c^4(\bR)$. Set $f_n := f_n^{\varepsilon,+}$ as in Lemma \ref{lemma:properties_fepsilon}.  We show that $(f(x) + \varepsilon(y^2 + \log(1+x^2)),Hf(x)+\tfrac{\varepsilon }{2} \vn{f'} + \varepsilon^2)$ is approximated by $(f_n,H_nf_n)$ as in Definition \ref{definition:subLIM_superLIM}(a). Since \eqref{eqn:convergence_condition_sublim_constants} was proved in Lemma~\ref{lemma:properties_fepsilon}\eqref{lemma:item:fepsilon_to_e}, we are left to check conditions \eqref{eqn:convergence_condition_sublim_uniform_gn} and \eqref{eqn:sublim_generators_upperbound}.
\begin{itemize}
\item[{\footnotesize \eqref{eqn:convergence_condition_sublim_uniform_gn}}]
We start by showing that we can get a uniform (in $n$) upper bound for the function $H_n f_n^{\varepsilon,+}$. To avoid trivialities, we consider the sequence for $n \geq N_{\star}$.
\begin{itemize}
\item 
If $\vert F_{n,f}(x,y) +\varepsilon \left(y^2 + F_{n,g}(x,y) \right)\vert \geq \sigma^2 \log b_n^{1/2}$, then the function $f_n^{\varepsilon,+}$ is constant and therefore $H_n f_n^{\varepsilon,+} \equiv 0$. 
\item
If $\vert F_{n,f}(x,y) +\varepsilon \left(y^2 + F_{n,g}(x,y) \right)\vert < \sigma^2 \log b_n^{1/2}$, the variables $x$ and $y$ are at most of order $b_n^{1/4}$ and $\log^{1/2} b_n^{1/2}$ respectively and we can characterize $H_n f_n^{\varepsilon,+}$ by means of \eqref{eqn:Hamiltonian_expansion}, since we can control the remainder term. Indeed,
\begin{itemize}
\item 
by Lemma \ref{lemma:control_epsilon}, we control $\varepsilon_n(y)$ up to $y$'s of order $\log^{1/2} b_n^{1/2}$;
\item 
the function $f$ is constant outside a compact set and thus has zero derivatives outside such a compact set;
\item by smoothness of the cut-off \eqref{eqn:cut-off}, the derivatives $\chi_n'$ and $\chi_n''$ are bounded.
\end{itemize}

We thus find
\begin{multline}\label{eqn:Hnfn_epsilon_expansion}
H_n f_n^{\varepsilon,+} (x,y) = \left[ -\frac{x^3}{2\sigma^4} f'(x) - \varepsilon \left( \frac{x^4}{\sigma^4(1+x^2)} + \frac{2 b_n^2 y^2}{\sigma^2}\right) \right] \chi_n'(-) \\
+ \frac{1}{2} \left[ \left( f'(x)\right)^2 + \frac{4 \varepsilon^2 x^2}{(1+x^2)^2} \right] \left( \chi_n'(-) \right)^2 \\
+ 8 \varepsilon^2 \sigma^2 y^2 \left[ \frac{b_n^4}{n} \chi_n''(-) + \left( \chi_n'(-) \right)^2 \right] +  Q_n(x,y)
\end{multline}
and $\sup_{(x,y)} \left\vert Q_n(x,y) \right\vert \leq c_0$, for a suitable positive constant $c_0$, independent of $n$ and $\varepsilon$. Observe that the remainder term $Q_n(x,y)$ collects all the smaller order contributions coming from $F_{n,f}(x,y)$, $F_{n,g}(x,y)$ and $y^2$. \\
We want to show that \eqref{eqn:Hnfn_epsilon_expansion} is uniformly bounded from above. The terms involving $f$ are ok, since $f \in C_c^4(\mathbb{R})$ implies that there exists a positive constant $M$ such that $f'$ vanishes at \mbox{$x \notin [-M,M]$}. The function $-\frac{x^4}{\sigma^4(1+x^2)}$ is non-positive and the term $\frac{2x^2}{(1+x^2)^2}$ is bounded from above by $2$. Moreover, if we set
\begin{equation}\label{def:N_2}
N_2 := \sup \left\{ n \in \mathbb{N} \left\vert \, -\frac{2b_n^2}{\sigma^2} + 8 \sigma^2 \left[ \frac{b_n^4}{n} \chi_n''(-) + \big( \chi_n'(-) \big)^2 \right] > 0 \right.\right\},
\end{equation}
we obtain that $-\frac{2 b_n^2 y^2}{\sigma^2} + 8 \sigma^2 \left[ \frac{b_n^4}{n} \chi_n''(-) + \big( \chi_n'(-) \big)^2 \right] y^2$ is uniformly bounded from above, for all $n > N_2$. By definition \eqref{def:N_star}, $N_{\star} \geq N_2$. Therefore, we can find a positive constant $c_1$ (dependent on $M$ and $\sigma$, but not on $n$ and $\varepsilon$) such that \mbox{$H_n f_n^{\varepsilon,+} (x,y) \leq c_1$}.
\end{itemize}

To conclude, observe that, since there exists a positive constant $c_2$ (independent of $n$) such that $\vert H_n f_n^{\varepsilon,+} \vert \leq c_2 \, b_n^2 \log b_n + c_0$ (cf. equation \eqref{eqn:Hnfn_epsilon_expansion}), choosing the sequence $v_n := b_n$ leads to $\sup_n v_n^{-1} \log \left\| H_n f_n^{\varepsilon,+} \right\| < \infty$.
\item[{\footnotesize \eqref{eqn:sublim_generators_upperbound}}]
Let $K$ be a compact set. Consider an arbitrary converging sequence $(x_n, y_n) \in K$ and let $(x,y) \in K$ be its limit. We want to show $\limsup_n H_n f_n^{\varepsilon,+}(x_n,y_n) \leq Hf(x)$.

As a converging sequence is bounded, by Lemma~\ref{lemma:properties_fepsilon}\eqref{lemma:item:fepsilon_to_d} we can find a sufficiently large $N=N(K) \in \mathbb{N}$ such that, for all $n \geq N$, we have 
\[
f_n^{\varepsilon,+} (x_n,y_n) = F_{n,f}(x_n,y_n) + \varepsilon \left( y_n^2 + F_{n,g}(x_n,y_n) \right).
\]
Thus, for any $n \geq N$, it yields 
\begin{multline*}
H_n f_n^{\varepsilon,+}(x_n,y_n) \leq -\frac{x_n^3}{2\sigma^4} f'(x_n) + \frac{1}{2} \left( f'(x_n) \right)^2\\
%
%
+ \varepsilon \left( \frac{x_n}{1+x_n^2} f'(x_n) -\frac{x_n^4}{\sigma^4(1+x_n^2)} - \frac{2 b_n^2 y_n^2}{\sigma^2} \right)  \\
%
%
+ \varepsilon^2 \left( \frac{2x_n^2}{(1+x_n^2)^2} + 8 \sigma^2 y_n^2 \right) + Q_n(x_n,y_n).
\end{multline*}
Using that $x(1+x^2)^{-1} \leq 1/2$, we find
\begin{equation*}
H_n f_n^{\varepsilon,+}(x_n,y_n) \leq Hf(x_n) +  \frac{\varepsilon}{2} \vn{f'} + \varepsilon^2 + \varepsilon y_n^2 \left[ 8 \varepsilon \sigma^2 - \frac{2 b_n^2}{\sigma^2}\right] + Q_n(x_n,y_n),
\end{equation*}
where the remainder term $Q_n$ converges to zero uniformly on compact sets. Since $b_n \uparrow \infty$, the conclusion follows.
\end{itemize}
\end{proof}

At this point we are ready to complete the definition of the sequences $\{f_n^{\varepsilon,\pm}\}_{n \in \mathbb{N}^*}$ by defining the index $N_{\star}$. We set
\begin{equation}\label{def:N_star}
N_{\star} := N_1 \vee N_2,
\end{equation}
with $N_1$ and $N_2$ given respectively in Lemma~\ref{lmm:cutoff_functions}\eqref{item_2:lmm:cutoff_functions} and in \eqref{def:N_2}. To conclude this section we obtain the Hamiltonian extensions.

\begin{proposition} \label{proposition:subsuper_extensions}
Consider notation as in Definition~\ref{definition:limiting_upper_lower_hamiltonians}. Moreover, set $\hat{H}_\dagger := H_\dagger \cup H$ and $\hat{H}_\ddagger := H_\ddagger \cup H$. Then $\hat{H}_\dagger$ is a sub-extension of $H_\dagger$ and $\hat{H}_\ddagger$ is a super-extension of $H_\ddagger$.
\end{proposition}

\begin{proof}
We prove only that $\hat{H}_\dagger$ is a sub-extension of $H_\dagger$. We use the first statement of Lemma \ref{lemma:extension_results}. Let $f \in \cD(H)$. We must show that $(f,Hf)$ is appropriately approximated  by elements in the graph of $H_\dagger$.\\ 
For any $n \in \mathbb{N}^*$, set $\varepsilon(n) = n^{-1}$ and consider the function $f_n (x,y) = f(x) + \varepsilon(n) \big( y^2 + \log(1+x^2) \big)$, with $H_\dagger f_n = Hf+ \frac{\varepsilon}{2} \|f'\|+\varepsilon^2$. From Lemma \ref{lemma:properties_fepsilon}\eqref{lemma:item:fepsilon_to_f}  we obtain that $\vn{f_n \wedge c - f \wedge c} \rightarrow 0$, for every $c \in \mathbb{R}$. In addition, as $Hf \in C_b(\bR)$, we have $\vn{H_\dagger f_n - Hf} \rightarrow 0$. This concludes the proof.
\end{proof}

\subsection{Exponential compact containment}
\label{Subsct:exponential_compact_containment}

The last open question we must address consists in verifying exponential compact containment for the fluctuation process. The validity of the compactness condition will be shown in Proposition~\ref{proposition:exponential_compact_containment} below.
We start with an informal discussion on the validity of this property. \\
Recall that the sequence of processes $\left(\sqrt{n} \left(n^{-1}T_n(t) - \sigma^2\right), t \geq 0 \right)$ converges weakly to the solution of \eqref{eqn:OU_process}. Thus, speeding up time by a factor $b_n^2$, we find that the process $\sqrt{n} \left(n^{-1}T_n(t) - \sigma^2\right)$ has roughly equilibrated as a centered normal random variable with variance $2 \sigma^4$. This implies that, for any $a > 0$, the tail probability $\PR\left[b_n \left(n^{-1}T_n(t) - \sigma^2\right) \geq a\right]$ scales like
	\begin{equation}\label{eqn:Gaussian_tail}
	\int_a^\infty \frac{1}{2\sigma^2\sqrt{\pi}} \frac{\sqrt{n}}{b_n} e^{- \frac{n}{b_n^2}\frac{y^2}{4 \sigma^4}} \dd y.
	\end{equation}
	By Lemma 2 in \cite[Sect.~7.1]{Fel68}, \eqref{eqn:Gaussian_tail} is bounded above by
	\begin{equation*}
	\frac{1}{2a\sigma^2 \sqrt{\pi}} \frac{b_n}{\sqrt{n}} \exp \left\{- \frac{n}{b_n^2} \frac{a^2}{4\sigma^4}\right\},
	\end{equation*}
which is indeed decaying on an exponential scale that is faster than $n b_n^{-4}$. As a consequence, it is the dynamics of the process $\left( b_n n^{-1} S_n(b_n^2 t), t \geq 0 \right)$ that needs to be properly controlled, as well as the interplay between the two processes.
	
To do so, we use a Lyapunov argument based on \cite[Lem.~4.22]{FeKu06} (included for completeness as Lemma \ref{lemma:compact_containment_FK}). We start by proving an auxiliary lemma showing that the function $(x,y) \mapsto \frac{1}{2}\left(y^2 + \log (1 +x^2)\right)$ is appropriate for this purpose, whenever we carry out the appropriate perturbation and cut-off as in the previous section.

\begin{lemma}\label{lemma:compact_containment_in_n}
Let $G \subseteq \mathbb{R}^2$ be a relatively compact open set. Consider the cut-off introduced in \eqref{eqn:cut-off} and define
\[
\Upsilon_n (x,y) = \chi_n \left[ \frac{1}{2} \left( y^2 + F_{n, g} (x,y) \right) \right],
\]
with $F_{n,\bullet}$ as in \eqref{eqn:perturbation_definition}, \eqref{eqn:perturbation_functions} and  $g(x) = \log(1+x^2)$. Then, we have
\[
\limsup_{n \uparrow \infty} \sup_{(x,y) \in G \cap E_n} H_n \Upsilon_n (x,y) \leq \frac{1}{4}.
\]
\end{lemma}

\begin{proof}
This follows immediately from the statement $H_\dagger \subseteq ex-\subLIM_n H_n$ proved in Proposition~\ref{proposition:limiting_hamiltonians}. Namely, one can consider $f \equiv 0$ and $\varepsilon = \frac{1}{2}$. 
\end{proof}

\begin{proposition} \label{proposition:exponential_compact_containment}
Assume that the sequence $(b_n n^{-1} S_n(0), b_n (n^{-1} T_n(0) - \sigma^2))$ is exponentially tight at speed $nb_n^{-4}$. Then, the processes 
\[
((X_n(t),Y_n(t)),t \geq 0) := \left(\left(b_n n^{-1} S_n(b_n^2t), b_n \left( n^{-1}T_n(b_n^2 t) - \sigma^2 \right)\right), t \geq 0 \right)
\] 
satisfy the exponential compact containment condition at speed $nb_n^{-4}$. In other words, for every compact set $K \subseteq \bR^2$, every constant $a \geq 0$ and time $T \geq 0$, there exists a compact set \mbox{$K' = K'(K,a,T) \subseteq \bR^2$} such that
\begin{equation*}
\limsup_{n \uparrow \infty} \sup_{(x,y) \in K \cap E_n} b_n^4 n^{-1} \log \PR \left[ (X_n(t),Y_n(t)) \notin K' \text{ for some } t \leq T \, | \, (X_n(0),Y_n(0)) = (x,y) \right] \leq - a.
\end{equation*}
\end{proposition}

\begin{proof}
The statement follows from Lemmas~\ref{lemma:compact_containment_in_n} and~\ref{lemma:compact_containment_FK} by choosing $f_n \equiv \Upsilon_n$ on a fixed, sufficiently large, compact set of $\mathbb{R}^2$. For similar proofs see e.g. \cite[Lem.~3.2]{DFL11}, \cite[Prop.~A.15]{CoKr17}.
\end{proof}

\subsection{Proof of Theorem~\ref{thm:path_space_MDP}}
\label{subsct:proof_main_theorem}

We check the assumptions of Theorem \ref{theorem:Abstract_LDP}. We use operators $H_\dagger$, $H_\ddagger$ as in Definition \ref{definition:limiting_upper_lower_hamiltonians} and limiting Hamiltonian $H \subseteq C_b(\bR) \times C_b(\bR)$, with domain $C_c^\infty(\bR)$, of the form $Hf(x) = H(x,f'(x))$ where
\[
H(x,p) = - \frac{x^3}{2\sigma^4}p + \frac{1}{2}p^2.
\]
We first verify Condition \ref{condition:H_limit}: \eqref{condition:H_limit:item_1} follows from Proposition~\ref{proposition:limiting_hamiltonians}, \eqref{condition:H_limit:item_2} is satisfied by definition and \eqref{condition:H_limit:item_3} follows from Proposition~\ref{proposition:subsuper_extensions}.

The comparison principle for $f- \lambda Hf = h$ for $h \in C_b(\bR)$ and $\lambda > 0$ has been verified in e.g. \cite[Prop.~3.5]{CoKr17}. Two things should be noted. The statement of the latter proposition is valid for $ f \in C_c^2(\bR)$, but the result generalizes straightforwardly to class $C_c^\infty(\bR)$ as the penalization and containment functions used in the proof are $C^\infty$. In addition, the proposition was stated for \textit{strong} viscosity solutions, but the proof of \cite[Prop.~3.5]{CoKr17} works for our notion of viscosity solutions as well. See the discussion following \cite[Def.~6.1 and Def.~7.1]{FeKu06} on the difference of the two notions of solutions.

Finally, the exponential compact containment condition follows from Proposition \ref{proposition:exponential_compact_containment}.

\appendix

\section{Appendix: path-space large deviations for a projected process} \label{appendix:large_deviations_for_projected_processes}

We turn to the derivation of the large deviation principle. We first introduce our setting.

\begin{assumption} \label{assumption:LDP_assumption}
Assume that, for each $n \in \mathbb{N}^*$, we have a Polish subset $E_n \subseteq \bR^2$ such that for each $x \in \bR^2$ there are $x_n \in E_n$ with $x_n \rightarrow x$.  Let $A_n \subseteq C_b(E_n) \times C_b(E_n)$ and existence and uniqueness holds for the $D_{E_n}(\bR^+)$ martingale problem for $(A_n,\mu)$ for each initial distribution $\mu \in \cP(E_n)$. Letting $\PR_{z}^n \in \cP(D_{E_n}(\bR^+))$ be the solution to $(A_n,\delta_z)$, the mapping $z \mapsto \PR_z^n$ is measurable for the weak topology on $\cP(D_{E_n}(\bR^+))$. Let $Z_n$ be the solution to the martingale problem for $A_n$ and set
\begin{equation*}
H_n f = \frac{1}{r_n} e^{-r_n f}A_n e^{r_n f} \qquad e^{r_n f} \in \cD(A_n),
\end{equation*}
for some sequence of speeds $(r_n)_{n \in \mathbb{N}^*}$, with $\lim_{n \uparrow \infty} r_n = \infty$.
\end{assumption}

Following the strategy of \cite{FeKu06}, the convergence of Hamiltonians $(H_n)_{n \in \mathbb{N}^*}$ is a major component in the proof of the large deviation principle. We postpone the discussion on how determining a limiting Hamiltonian $H$ due to the difficulties that taking the $n \uparrow \infty$ limit introduces in our particular context. We first focus on exponential tightness, an equally important aspect. 

\subsection{Compact containment condition}

Given the convergence of the Hamiltonians, to have exponential tightness it suffices to establish an exponential compact containment condition.

\begin{definition} \label{definition:compact_containment}
	We say that a sequence of processes $(Z_n(t),t \geq 0)$ on $E_n \subseteq \bR^2$ satisfies the exponential compact containment condition at speed $(r_n)_{n \in \mathbb{N}^*}$, with $\lim_{n \uparrow \infty} r_n = \infty$, if for all compact sets $K\subseteq \bR^2$,  constants $a \geq 0$ and times $T > 0$, there is a compact set $K' \subseteq \bR^2$ with the property that
	\begin{equation*}
	\limsup_{n \uparrow \infty} \sup_{z \in K} \frac{1}{r_n} \log \PR\left[Z_n(t) \notin K' \text{ for some } t \leq T \, \middle| \, Z_n(0) = z\right] \leq - a.
	\end{equation*}
\end{definition}

The exponential compact containment condition can be verified by using approximate Lyapunov functions and martingale methods. This is summarized in the following lemma. Note that exponential compact containment can be obtained by taking deterministic initial conditions.

\begin{lemma}[Lemma 4.22 in \cite{FeKu06}] \label{lemma:compact_containment_FK}
Suppose Assumption \ref{assumption:LDP_assumption} is satisfied. Let $Z_n(t)$ be a solution of the martingale problem for $A_n$ and assume that $(Z_n(0))_{n \in \mathbb{N}^*}$ is exponentially tight with speed $(r_n)_{n \in \mathbb{N}^*}$. Consider the compact set $K = [a,b] \times [c,d]$ and let $G \subseteq \mathbb{R}^2$ be open and such that $[a,b]\times [c,d] \subseteq G$. For each $n$, suppose we have $(f_n,g_n) \in H_n$. Define
	\begin{align*}
	\beta(q,G) &:= \liminf_{n \uparrow \infty} \left( \inf_{(x,y) \in G^c} f_n(x,y) - \sup_{(x,y) \in K} f_n(x,y)\right), \\
	\gamma(G) & := \limsup_{n \uparrow \infty} \sup_{(x,y) \in G} g_n(x,y).
	\end{align*}
	Then
	\begin{multline*} 
	\limsup_{n \uparrow \infty} \frac{1}{r_n} \log \PR\left[Z_n(t) \notin G \text{ for some } t \leq T  \right] \\
	\leq \max \left\{-\beta(q,G) + T \gamma(G), \limsup_{n \uparrow \infty} \PR\left[Z_n(0) \notin [a,b]\times[c,d] \right] \right\}.
	\end{multline*}
\end{lemma}

\subsection{Operator convergence for a projected process}

In the papers \cite{Kr16b,CoKr17,DFL11} one of the main steps in proving the large deviation principle was proving directly the existence of an operator $H$ such that $H \subseteq \LIM_n H_n$; in other words, verifying that, for all $(f,g) \in H$, there are $f_n \in H_n$ such that $\LIM_n f_n = f$ and $\LIM_n H_n f_n = g$ (the notion of $\LIM$ is introduced in Definition~\ref{def:LIM}). Here it is hard to follow a similar strategy. We rather proceed as done in \cite{CoKr18}.\\
We are dealing with functions 
\[
f_n(x,y) = f(x) + b_n^{-1} f_1(x,y) + b_n^{-2} f_2(x,y) \qquad \mbox{ (for suitably chosen $f_1$ and $f_2$)}
\]
given in a perturbative fashion and satisfying intuitively $f_n \rightarrow f$ and $H_n f_n \rightarrow Hf$ with Hamiltonian $H \subseteq C_b(\bR) \times C_b(\bR)$ of the form \eqref{eqn:limiting_Hamiltonian}. The unboundedness of the state space $E_n$ causes that for most functions $f \in C_c^4(\mathbb{R})$: $\sup_n \vn{f_n} = \infty$, implying we do not have $\LIM f_n = f$. To circumvent this issue, we relax our definition of limiting operator.

In particular, we will work with two Hamiltonians $H_{\dagger}$ and $H_\ddagger$, that are limiting upper and lower bounds for the sequence of Hamiltonians $H_n$, respectively, and thus serve as natural upper and lower bounds for $H$. This extension allows us to consider unbounded functions in the domain and to argue with inequalities rather than equalities.

\begin{definition}[Definition 2.5 in \cite{FeKu06}]\label{def:LIM}
For $f_n \in C_b(E_n)$ and $f \in C_b(\bR^2)$, we will write $\LIM f_n = f$ if $\sup_n \vn{f_n} < \infty$ and, for all compact sets $K \subseteq \bR^2$,
\begin{equation*}
\lim_{n \uparrow \infty} \sup_{(x,y) \in K \cap E_n} \left|f_n(x,y) - f(x,y) \right| = 0.
\end{equation*}
\end{definition}

\begin{definition}[Condition 7.11 in \cite{FeKu06}] \label{definition:subLIM_superLIM}
Suppose that for each $n$ we have an operator $H_{n} \subseteq C_b(E_n) \times C_b(E_n)$. 
Let $(v_n)_{n \in \mathbb{N}^*}$ be a sequence of real numbers such that $v_n \uparrow \infty$.

\begin{enumerate}[(a)]
	\item The \textit{extended sub-limit}, denoted by $ex-\subLIM_n H_{n}$, is defined by the collection $(f,g) \in C_l(\bR^2)\times C_b(\bR)$ for which there exist $(f_n,g_n) \in H_{n}$ such that
	\begin{gather} 
	\LIM f_n \wedge c = f \wedge c, \qquad \forall \, c \in \bR, \label{eqn:convergence_condition_sublim_constants} \\
	\sup_n \frac{1}{v_n} \log \vn{g_n} < \infty, \qquad \sup_{n} \sup_{x \in \bR^2} g_n(x) < \infty, \label{eqn:convergence_condition_sublim_uniform_gn}
	\end{gather}
	and that, for every compact set $K \subseteq \bR^2$ and every sequence $z_n \in K$ satisfying $\lim_n z_n = z$ and $\lim_n f_n(z_n) = f(z) < \infty$, 
	\begin{equation} \label{eqn:sublim_generators_upperbound}
	\limsup_{n \uparrow \infty}g_{n}(z_n) \leq g(z).
	\end{equation}
	\item 
	The \textit{extended super-limit}, denoted by $ex-\superLIM_n H_{n}$, is defined by the collection $(f,g) \in C_u(\bR^2)\times C_b(\bR)$ for which there exist $(f_n,g_n) \in H_{n}$ such that
	\begin{gather} 
	\LIM f_n \vee c = f \vee c, \qquad \forall \, c \in \bR, \label{eqn:convergence_condition_superlim_constants} \\
	\sup_n \frac{1}{v_n} \log \vn{g_n} < \infty, \qquad \inf_{n} \inf_{x \in \bR^2} g_n(x) > - \infty, \label{eqn:convergence_condition_superlim_uniform_gn}
	\end{gather}
	and that, for every compact set $K \subseteq \bR^2$ and every sequence $z_n \in K$ satisfying $ \lim_n z_n = z$ and $\lim_n f_n(z_n) = f(z) > - \infty$,
	\begin{equation}\label{eqn:superlim_generators_lowerbound}
	\liminf_{n \uparrow \infty}g_{n}(z_n) \geq g(z).
	\end{equation}
\end{enumerate}
\end{definition}

For completeness, we also give the definition of the extended limit.

\begin{definition}\label{def:extended_limit}
Suppose that for each $n$ we have an operator $H_{n} \subseteq C_b(E_n) \times C_b(E_n)$. We write $ex-\LIM H_n$ for the set of $(f,g) \in C_b(\bR^2) \times C_b(\bR^2)$ for which there exist $(f_n,g_n) \in H_n$ such that $f = \LIM f_n$ and $g = \LIM g_n$.
\end{definition}

\begin{definition}[Viscosity solutions]\label{def:viscosity_solution}
Let $H_\dagger \subseteq C_l(\bR^2) \times C_b(\bR^2)$ and $H_\ddagger \subseteq C_u(\bR^2) \times C_b(\bR^2)$ and let $\lambda > 0$ and $h \in C_b(\bR^2)$. Consider the Hamilton-Jacobi equations
\begin{align}
f - \lambda H_\dagger f & = h, \label{eqn:differential_equation_dagger} \\
f - \lambda H_\ddagger f & = h. \label{eqn:differential_equation_ddagger}
\end{align}
We say that $u$ is a \textit{(viscosity) subsolution} of equation \eqref{eqn:differential_equation_dagger} if $u$ is bounded, upper semi-continuous and if, for every $f \in \cD(H_\dagger)$ such that $\sup_x u(x) - f(x) < \infty$ there exists a sequence  $x_n \in \mathbb{R}^2$ such that
\begin{equation*}
\lim_{n \uparrow \infty} u(x_n) - f(x_n)  = \sup_x u(x) - f(x),
\end{equation*}
and 
\begin{equation*}
\lim_{n \uparrow \infty} u(x_n) - \lambda H_\dagger f(x_n) - h(x_n) \leq 0.
\end{equation*}
We say that $v$ is a \textit{(viscosity) supersolution} of equation \eqref{eqn:differential_equation_ddagger} if $v$ is bounded, lower semi-continuous and if, for every $f \in \cD(H_\ddagger)$ such that $\inf_x v(x) - f(x) > - \infty$ there exists a sequence $x_n \in \mathbb{R}^2$ such that
\begin{equation*}
\lim_{n \uparrow \infty} v(x_n) - f(x_n)  = \inf_x v(x) - f(x),
\end{equation*}
and 
\begin{equation*}
\lim_{n \uparrow \infty} v(x_n) - \lambda H_\ddagger f(x_n) - h(x_n) \geq 0.
\end{equation*}
We say that $u$ is a \textit{(viscosity) solution} of equations \eqref{eqn:differential_equation_dagger} and \eqref{eqn:differential_equation_ddagger} if it is both a subsolution to \eqref{eqn:differential_equation_dagger} and a supersolution to \eqref{eqn:differential_equation_ddagger}.

We say that \eqref{eqn:differential_equation_dagger} and \eqref{eqn:differential_equation_ddagger} satisfy the \textit{comparison principle} if for every subsolution $u$ to \eqref{eqn:differential_equation_dagger} and supersolution $v$ to \eqref{eqn:differential_equation_ddagger}, we have $u \leq v$.
\end{definition}

Note that the comparison principle implies uniqueness of viscosity solutions. This in turn implies that a new Hamiltonian can be constructed based on the set of viscosity solutions.

\begin{condition} \label{condition:H_limit}
Suppose we are in the setting of Assumption \ref{assumption:LDP_assumption}. Suppose there are operators $H_\dagger \subseteq C_l(\bR^2) \times C_b(\bR^2)$, $H_\ddagger \subseteq C_u(\bR^2) \times C_b(\bR^2)$ and $H \subseteq C_b(\bR) \times C_b(\bR)$ with the following properties:
\begin{enumerate}[(a)]
\item \label{condition:H_limit:item_1} 
$H_\dagger \subseteq ex-\subLIM_n H_n$ and $H_\ddagger \subseteq ex-\superLIM_n H_n$ (recall Definition~\ref{definition:subLIM_superLIM}).
\item \label{condition:H_limit:item_2} 
The domain $\cD(H)$ contains $C^\infty_c(\mathbb{R})$ and, for $f \in C_c^\infty(\bR)$, we have $Hf(x) = H(x,\nabla f(x))$.
\item \label{condition:H_limit:item_3}
For all $\lambda > 0$ and $h \in C_b(\bR)$, every subsolution to $f - \lambda H_\dagger f = h$ is a subsolution to $f - \lambda H f = h$ and every supersolution to $f - \lambda H_\ddagger f = h$ is a supersolution to $f - \lambda H f = h$.
\end{enumerate}
\end{condition}

Now we are ready to state the main result of this appendix: the large deviation principle for the projected process. We denote by $\eta_n: E_n \to \mathbb{R}$ the projection map $\eta_n(x,y)=x$.

\begin{theorem}[Large deviation principle] \label{theorem:Abstract_LDP}
Suppose we are in the setting of Assumption \ref{assumption:LDP_assumption} and Condition \ref{condition:H_limit} is satisfied. Suppose that for all $\lambda > 0$ and $h \in C_b(\mathbb{R})$ the comparison principle holds for \mbox{$f - \lambda H f = h$}.   

Let $Z_n(t)$ be the solution to the martingale problem for $A_n$. Suppose that the large deviation principle at speed $(r_n)_{n \in \mathbb{N}^*}$ holds for $\eta_n(Z_n(0))$ on $\bR$ with good rate-function $I_0$. Additionally suppose that the exponential compact containment condition holds at speed $(r_n)_{n \in \mathbb{N}^*}$ for the processes $Z_n(t)$.

Then the large deviation principle holds with speed $(r_n)_{n \in \mathbb{N}^*}$  for $(\eta_n(Z_n(t)))_{n \in \mathbb{N}^*}$ on $D_{\mathbb{R}}(\bR^+)$ with good rate function $I$. Additionally, suppose that the map $p \mapsto H(x,p)$ is convex and differentiable for every $x$ and that the map $(x,p) \mapsto \frac{\dd}{\dd p} H(x,p)$ is continuous. Then the rate function $I$ is given by
\begin{equation*}
I(\gamma) = \begin{cases}
I_0(\gamma(0)) + \int_0^\infty \cL(\gamma(s),\dot{\gamma}(s)) \dd s & \text{if } \gamma \in \cA\cC, \\
\infty & \text{otherwise},
\end{cases}
\end{equation*}
where $\cL : \mathbb{R}^2 \rightarrow \bR$ is defined by $\cL(x,v) = \sup_p \left\{pv - H(x,p)\right\}$.
\end{theorem}

\begin{proof}
The large deviation result follows by \cite[Cor.~8.28]{FeKu06} with $H_\dagger$ and $H_\ddagger$ as in the present paper and $\bfH_\dagger = \bfH_\ddagger = H$. The verification of the conditions for \cite[Thm.~8.27]{FeKu06} corresponding to a Hamiltonian of this type have been carried out in e.g. \cite[Sect.~10.3]{FeKu06} or \cite{CoKr17}.
\end{proof}

\subsection{Relating two sets of Hamiltonians}

For Condition \ref{condition:H_limit}, we need to relate the Hamiltonians $H_\dagger \subseteq C_l(\bR^2) \times C_b(\bR^2)$ and $H_\ddagger \subseteq C_u(\bR^2) \times C_b(\bR^2)$ to $H \subseteq C_b(\bR) \times C_b(\bR)$. 

\begin{definition}
Let $H_\dagger \subseteq C_l(\bR^2) \times C_b(\bR^2)$ and $H_\ddagger \subseteq C_u(\bR^2) \times C_b(\bR^2)$. We say that $\hat{H}_\dagger \subseteq  C_l(\bR^2) \times C_b(\bR^2)$ is a \textit{viscosity sub-extension} of $H_\dagger$ if $H_\dagger \subseteq \hat{H}_\dagger$ and if for every $\lambda >0$ and $h \in C_b(\bR^2)$ a viscosity subsolution to $f-\lambda H_\dagger f = h$ is also a viscosity subsolution to $f - \lambda \hat{H}_\dagger f = h$. Similarly, we define a \textit{viscosity super-extension} $\hat{H}_\ddagger$ of $H_\ddagger$.
\end{definition}

The following lemma allows us to obtain viscosity extensions.

\begin{lemma}[Lemma 7.6 in \cite{FeKu06}] \label{lemma:extension_results}
Let $H_\dagger \subseteq \hat{H}_\dagger \subseteq C_l(\bR^2) \times C_b(\bR^2)$ and $H_\ddagger \subseteq \hat{H}_\ddagger \subseteq C_u(\bR^2) \times C_b(\bR^2)$.

\smallskip

Suppose that for each $(f,g) \in \hat{H}_\dagger$ there exist $(f_n,g_n) \in H_\dagger$ such that, for every $c,d \in \bR$, we have
\begin{equation*}
\lim_{n \uparrow \infty} \vn{f_n \wedge c - f \wedge c} = 0
\end{equation*}
and 
\begin{equation*}
\limsup_{n \uparrow \infty} \sup_{z :  f(\gamma(z)) \vee f_n(\gamma(z)) \leq c} g_n(z) \vee d - g(z) \vee d \leq 0.
\end{equation*}
Then $\hat{H}_\dagger$ is a sub-extension of $H_\dagger$.

\smallskip

Suppose that for each $(f,g) \in \hat{H}_\ddagger$ there exist $(f_n,g_n) \in H_\ddagger$ such that, for every $c,d \in \bR$, we have
\begin{equation*}
\lim_{n \uparrow \infty} \vn{f_n \vee c - f \vee c} = 0
\end{equation*}
and 
\begin{equation*}
\liminf_{n \uparrow \infty} \inf_{z :  f(\gamma(z)) \wedge f_n(\gamma(z)) \geq c} g_n(z) \wedge d - g(z) \wedge d \geq 0.
\end{equation*}
Then $\hat{H}_\ddagger$ is a super-extension of $H_\ddagger$.
\end{lemma}

\smallskip

\textbf{Acknowledgements}
The authors are grateful to the anonymous referees who pointed out important points which led to an improvement of the paper.
FC was supported by The Netherlands Organisation for Scientific Research (NWO) via TOP-1 grant 613.001.552 and by the French foundation {\em Fondation Sciences Ma\-th{\'e}\-ma\-ti\-ques de Paris}. Part of this work was done during FC's stay at the Institut Henri Poincar{\'e} -- Centre {\'E}mile Borel during the trimester {\em Stochastic Dynamics Out of Equilibrium}. FC thanks this institution for hospitality. 
MG was supported by the French foundation \emph{Fondation Math\'ematique Jacques Hadamard} and partially by the Dutch cluster {\em Stochastics -- Theoretical and Applied Research} (STAR). 
RK was supported by the Deutsche Forschungsgemeinschaft (DFG) via RTG 2131 High-dimensional Phenomena in Probability – Fluctuations and Discontinuity.

\bibliographystyle{abbrv} 
\bibliography{../MDP_SOC_CW_biblio}{}
\end{document}